\RequirePackage{fix-cm,amsmath}
\documentclass[smallextended]{svjour3}       
\usepackage{graphicx}
\usepackage{fancyhdr, amssymb, mathtools, algpseudocode, algorithm, nicefrac, multirow, booktabs, comment}
\mathtoolsset{showonlyrefs} 
\usepackage[dvipsnames]{xcolor}
\usepackage[shortlabels]{enumitem}
\usepackage[letterpaper,top=2in, bottom=1.5in, left=1in, right=1in]{geometry}

\usepackage{tikz,pgfplots}
\pgfplotsset{compat=1.12}
\usepackage[strings]{underscore}

\newcommand{\R}{\mathbb{R}}

\newcommand{\X}{\mathcal{X}}

\renewcommand{\i}{_{i}}
\newcommand{\im}{_{i-1}}

\renewcommand{\k}{_{k}}
\newcommand{\km}{_{k-1}}
\newcommand{\kp}{_{k+1}}

\newcommand{\g}{\gamma}
\newcommand{\G}{\Gamma}
\renewcommand{\b}{\beta}
\newcommand{\gf}{\nabla f}

\newcommand\ang[1]{\left\langle#1\right\rangle}
\newcommand{\norm}[1]{\left\lVert#1\right\rVert}
\newcommand{\pr}[1]{\left(#1\right)}
\newcommand{\cbr}[1]{\left\{#1\right\}}
\newcommand{\br}[1]{\left[#1\right]}
\newcommand{\ceil}[1]{\left\lceil#1\right\rceil}

\newcommand{\MO}{\mathcal{O}}
\newcommand{\cX}{\mathcal{X}}

\begin{document}
	
	\title{Backtracking linesearch for conditional gradient sliding
	}
	
	
	\author{Hamid Nazari         \and
		Yuyuan Ouyang 
	}
	
	
	\institute{Hamid Nazari \at
		School of Mathematical and Statistical Sciences \\
		Clemson University\\
		\email{snazari@clemson.edu}           
		\and
		Yuyuan Ouyang \at
		School of Mathematical and Statistical Sciences \\
		Clemson University\\
		\email{yuyuano@clemson.edu}
	}
	
	\date{Received: date / Accepted: date}

	\maketitle
	
	\begin{abstract}
		We present a modification of the conditional gradient sliding (CGS) method that was originally developed in \cite{lan2016conditional}. While the CGS method is a theoretical breakthrough in the theory of projection-free first-order methods since it is the first that reaches the theoretical performance limit, in implementation it requires the knowledge of the Lipschitz constant of the gradient of the objective function $L$ and the number of total gradient evaluations $N$. Such requirements imposes difficulties in the actual implementation, not only because that it can be difficult to choose proper values of $L$ and $N$ that satisfies the conditions for convergence, but also since conservative choices of $L$ and $N$ can deteriorate the practical numerical performance of the CGS method. 
		Our proposed method, called the conditional gradient sliding method with linesearch (CGS-ls), does not require the knowledge of either $L$ and $N$, and is able to terminate early before the theoretically required number of iterations. While more practical in numerical implementation, the theoretical performance of our proposed CGS-ls method is still as good as that of the CGS method. 
		We present numerical experiments to show the efficiency of our proposed method in practice.
	\end{abstract}
	

	\section{Introduction}

The problem of interest of this paper is the convex optimization problem
\begin{align}
	\label{eq:problem}
	f^*:=\min_{x\in \cX}f(x)
\end{align}
where $\cX\in\R^n$ is a convex compact set and $f:\R^n\to\R$ is a convex differentiable function such that
\begin{align}
	\label{eq:L}
	\|\nabla f(x) - \nabla f(y)\|\le L\|x - y\|,\ \forall x,y\in\cX.
\end{align}
Here $\|\cdot\|$ is the Euclidean norm. 
Our goal is to compute an $\varepsilon$-approximate solution $y$ such that $f(y) - f^*\le\varepsilon$ using first-order information, namely, the function and gradient values $f$ and $\nabla f$. 

When designing numerical methods that uses first-order information, the structures of the objective function and the compact feasible set $\cX$ have significant impact on the theoretical convergence theory and practical computing performance. As an example, consider the performance of two classical iterative numerical methods on a simple instance of problem \eqref{eq:problem} in which $\cX:=\{(x^{(1)},x^{(2)})\in\R^2|x^{(1)}+x^{(2)}=1,x^{(1)},x^{(2)}\ge 0\}$ is a line segment in $\R^2$ and $f(x)=\|x\|^2/2$. The first numerical method is the projected gradient method with iterates
\begin{align}
	x_k = y\km - (1/\beta_k)\nabla f(y\km),\ y_k=\text{Proj}_{\cX}(x_k).
\end{align}
In each iteration, the projected gradient method moves along the negative gradient direction with a pre-specified stepsize $(1/\beta_k)$ to obtain an updated iterate $x_k$. Such update $x_k$ may fall outside of the feasible region, so we maintain feasibility by projecting $x_k$ onto the feasible set to obtain a new approximate solution $y_k$. The second numerical method is the conditional gradient method \cite{frank1956algorithm}
\begin{align}
	\label{eq:cg}
	x_k = \min_{x\in\cX}\langle\nabla f(x\km),x\rangle,\ y_k = \sum_{i=1}^{k}{\lambda^i_k}x_i.
\end{align}
Here we compute an update $x_k$ from a linear optimization over the line segment $\cX$. Note that $x_k$ will always be selected from the two extreme points $(1,0)^\top$ and $(0,1)^\top$. To avoid oscillating between the extreme points, we select the new approximate solution $y_k$ to a convex combination of all previous updates $x_i$'s. Both the aforementioned methods have been extensively studied in the literature. See, e.g., the monograph \cite{bubeck2015convex} or the book \cite{lan2020first} for the survey of both methods.

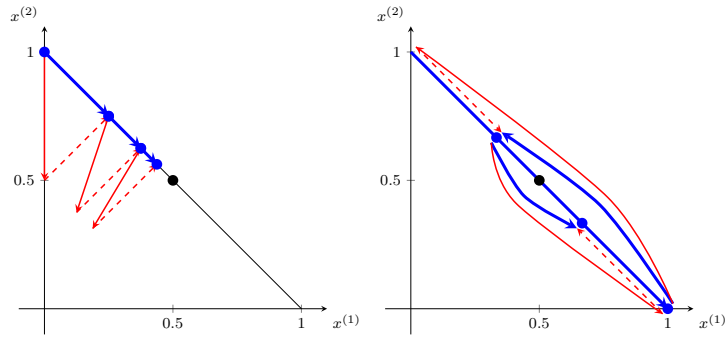
\begin{figure}[!h]
	\centering
	\scalebox{.72}{
		\begin{tikzpicture}[>=latex]
		\begin{axis}[
		unit vector ratio*=1 1 1,
		axis x line=center,
		axis y line=center,
		xtick={0,.5,1},
		ytick={0,.5,1},
		xlabel={$x^{(1)}$},
		ylabel={$x^{(2)}$},
		xlabel style={below right},
		ylabel style={above left},
		xmin=-.1,
		xmax=1.1,
		ymin=-0.1,
		ymax=1.1]
		\addplot [mark=none,domain=0:1] {1-x};
		\addplot [mark=*, ultra thick] coordinates {(.5, .5)};
		
		\addplot[color=red,-stealth, thick] coordinates {
			(0, 1) 
			(0, .5)
		};
		\addplot[color=red,-stealth, dashed, thick] coordinates {
			(0, .5)
			(.25, .75)
		};
		\addplot[color=red,-stealth, thick] coordinates {
			(.25, .75)
			(0.125, .375)
		};
		\addplot[color=red,-stealth, dashed, thick] coordinates {
			(0.125, .375)
			(.375, .625)
		};
		\addplot[color=red,-stealth, thick] coordinates {
			(.375, .625)
			(.1875, .3125)
		};
		\addplot[color=red,-stealth, dashed, thick] coordinates {
			(.1875, .3125)
			(.4375, .5625)
		};
		
		\addplot[color=blue,-stealth, ultra thick, mark=*] coordinates {
			(0, 1) 
			(.25, .75)
		};
		\addplot[color=blue,-stealth, ultra thick, mark=*] coordinates {
			(.25, .75)
			(.375, .625)
		};
		\addplot[color=blue,-stealth, ultra thick, mark=*] coordinates {
			(.375, .625)
			(.4375, .5625)
		};
		
		\end{axis}
		\end{tikzpicture}	
		\begin{tikzpicture}[>=latex]
		\begin{axis}[
		unit vector ratio*=1 1 1,
		axis x line=center,
		axis y line=center,
		xtick={0,.5,1},
		ytick={0,.5,1},
		xlabel={$x^{(1)}$},
		ylabel={$x^{(2)}$},
		xlabel style={below right},
		ylabel style={above left},
		xmin=-.1,
		xmax=1.1,
		ymin=-0.1,
		ymax=1.1]
		\addplot [mark=none,domain=0:1] {1-x};
		\addplot [mark=*, ultra thick] coordinates {(.5, .5)};
		
		\addplot[color=blue,-stealth, ultra thick, mark=none] coordinates {
			(0, 1) 
			(1, 0)
		};
		\addplot[color=blue,-stealth, ultra thick, mark=none, smooth] coordinates {
			(1.02, 0.02)
			(0.7367, 0.4033)
			(.3533, .6867)
		};
		\addplot[color=blue,-stealth, ultra thick, mark=none, smooth] coordinates {
			(.3133, .6467)
			(.44, .44)
			(.6467, 0.3133)
		};
		
		\addplot[color=red,-stealth, thick, mark=none, smooth] coordinates {
			(1.02, 0.02)
			(0.7667, 0.4333)
			(0.02, 1.02)
		};
		\addplot[color=red,dashed,-stealth, thick, mark=none] coordinates {
			(0.02, 1.02) 
			(0.3533, 0.6867)
		};
		\addplot[color=red,-stealth, thick, mark=none, smooth] coordinates {
			(.3133, .6467)
			(.41, .41)
			(.98, -.02)
		};
		\addplot[color=red,dashed,-stealth, thick, mark=none] coordinates {
			(.98, -.02)
			(0.6467, 0.3133)
		};
		
		\addplot [mark=*, ultra thick, color=blue] coordinates {(1, 0)};
		\addplot [mark=*, ultra thick, color=blue] coordinates {(.3333, .6667)};
		\addplot [mark=*, ultra thick, color=blue] coordinates {(.6667, 0.3333)};
		
		\end{axis}
		\end{tikzpicture}	
	}
	\caption{\label{fig:simple_plot} \footnotesize Left: Project gradient method iterations. Blue arrows describe update of approximate solutions, red solid arrows describe gradient descent moves, and red dashed arrows describe projections onto $\cX$. 
		Right: Conditional gradient method iterations. Blue arrows describe update of approximate solutions, red solid arrows describe the updates $x_k$ (always extreme points), and red dashed arrows describe the convex combination weighting process.
	}
\end{figure}

The performance of the projected gradient and conditional gradient methods are illustrated in Figure \ref{fig:simple_plot}. Both methods start at point $(0,1)^\top$. In the projected gradient method we select the constant stepsize $1/\beta_k\equiv 0.5$ and in the conditional gradient method we select the weights $\lambda_i = 2i/(k(k+1))$. From Figure \ref{fig:simple_plot} we can observe that approximate solutions of both methods are gradually approaching the optimal solution. The ones produced by the projected gradient method converge faster comparing to the ``back and forth'' behavior (see the blue arrows in the right plot) from the conditional gradient method; this is because the projected gradient method benefits from the structure of the objective function (i.e., smoothness and strong convexity). However, in the conditional gradient method iterations we only need to solve a linear program, while the projected gradient method needs to compute  projections onto $\cX$. If we extend the aforementioned simple example to the general case in where $\cX$ is a high dimensional general convex polytope, then solving linear programs would likely be more preferable than solving projections onto general convex polytopes. 

In fact, the above simple example reflects the difference between the projected gradient and conditional gradient methods in terms of their theoretical convergence theory. For our problem of interest \eqref{eq:problem}, the accelerated version of projected gradient method (also known as Nesterov's optimal method; see \cite{nesterov1983method,nesterov2004introductory}) is able to compute an $\varepsilon$-approximate solution with $\MO(\sqrt{L/\varepsilon})$ gradient evaluations and projections, and the conditional gradient method will need $\MO(L/\varepsilon)$ gradient evaluations and linear objective optimizations \cite{harchaoui2015conditional,jaggi2013revisiting}. Here $L$ is the Lipschitz constant defined in \eqref{eq:L}. It should be noted that both convergence properties are not improvable in the following sense: for any numerical method that uses only first-order information for solving problem \eqref{eq:problem} with large number of decision variables $n$, it requires at least $\MO(\sqrt{L/\varepsilon})$ gradient computation to compute an $\varepsilon$-solution \cite{nemirovski1992information} (see also \cite{nemirovski1994efficient,nesterov2004introductory}); for any numerical method that defines a linear objective function and solve a linear objective optimization in each iteration, it requires at least $\MO(L/\varepsilon)$ iterations to compute an $\varepsilon$-solution \cite{lan2013complexity} to problem \eqref{eq:problem} with large $n$.

The conditional gradient methods are still preferred in many practices due to its advantage of not requiring the projection computation, even though its requirement of $\MO(L/\varepsilon)$ gradient computations is not as good as the theoretical limit $\MO(\sqrt{L/\varepsilon})$. For problems with sophisticated feasible set $\cX$, the possibly expensive computational time of projection operator can significantly outweigh the theoretical advantage of the smaller $\MO(\sqrt{L/\varepsilon})$ gradient evaluation of any projection-based methods. 

Recently, there has been a breakthrough in the theory of conditional gradient methods. A condition gradient sliding (CGS) method is proposed in \cite{lan2016conditional} that is able to solve an $\varepsilon$-approximate solution of problem \eqref{eq:problem} with $\MO(\sqrt{L/\varepsilon})$ gradient evaluations and $\MO(L/\varepsilon)$ linear objective optimization subproblems. It should be pointed out that the CGS method still obeys the aforementioned theoretical performance limits of numerical methods that uses first-order information. In fact, it is the first conditional gradient-type method that reaches the theoretical performance limit. The key concept of the CGS method is to use the conditional gradient method to approximately solve the projection subproblem in the accelerated projected gradient method. With proper precision on solving the projection subproblem, the CGS method is able to reduce the number of gradient evaluations from the required $\MO(L/\varepsilon)$ of the original conditional gradient method to a improve order of $\MO(\sqrt{L/\varepsilon})$.

In the original CGS method in \cite{lan2016conditional}, the knowledge of the Lipschitz constant $L$ and the number of gradient evaluations $N$ are required for implementation\footnote{Some parameter settings of the CGS method does not require $N$; see, e.g., Corollary 2.3 in \cite{lan2016conditional}. However, since no termination criterion is proposed in \cite{lan2016conditional}, to terminate we still need to specify the total number of iterations $N$. It should be noted that we can use some termination criterion for the CGS method, e.g., the Wolfe gap, which we will use for CGS in the numerical experiments of our paper; however, the theoretical convergence property in terms of the number of iterations needed to achieve small Wolfe gap is different from the properties of the CGS method, and may deteriorate its practical performance significantly.}. Such requirements lead to two disadvantages in practice. First, in order to make sure that a constant $L$ satisfies the Lipschitz condition \eqref{eq:L}, we will need to choose a constant $L$ that satisfies the Lipschitz condition \eqref{eq:L} for all pairs $x$ and $y$ in $\cX$. Computing such $L$ can be difficult; the computed $L$ can also be too conservative and lead to worse practical performance. Second, in order to compute an $\varepsilon$-approximate solution, we need to either tune the number of gradient evaluations $N$ in practice or follow its theoretical property and specify a possibly conservative $N=\MO(L/\sqrt{\varepsilon})$. While the CGS method reaches the theoretical performance limits, such disadvantages may deteriorate its practical performance significantly. 

In this paper, we propose a modification of the CGS method that allows for its practical implementation. Our proposed method, called the CGS with linesearch (CGS-ls), performs a backtracking linesearch strategy to gradually increase the initial guess of Lipschitz constant $L_0$ to values that satisfy the convergence condition. The initial guess $L_0$ does not need to satisfy the Lipschitz condition \eqref{eq:L} and can be significantly smaller than the actual Lipschitz constant. We also maintains the estimate of a lower bound of $f^*$ that certificates the achievement of an $\varepsilon$-approximate solution. Consequently, our propose method does not require the knowledge of either $L$ and $N$ and is able to stop before the theoretical $\MO(\sqrt{L/\varepsilon})$ bound of required gradient evaluations or the $\MO(L/\varepsilon)$ bound of linear objective optimization subproblems. It should be noted that our theoretical analysis of the backtracking linesearch is non-trivial. In order to improve the practical implementation, add proper termination criterion, and maintain the same theoretical convergence properties as the CGS method, we need to modify some theoretical analysis in the original CGS results in \cite{lan2016conditional}. We demonstrate through numerical experiments the advantages of our proposed CGS-ls method in implementation.
	\section{Proposed Algorithm}

In this section, we describe our proposed conditional gradient sliding method with linesearch (CGS-ls) in Algorithm \ref{Alg:CGSLS}.

\begin{algorithm}[!ht]
	\caption{\label{Alg:CGSLS} A conditional gradient sliding algorithm with backtracking linesearch (CGS-ls)}
	\begin{algorithmic}
		\State Choose initial point $y_0\in \mathcal{X}$ and initial guess of Lipschitz constant $L_0>0$. Set $x_0=y_0$. Define function $\xi_0(x)\equiv 0$.
		\For {$k=1,2,\ldots,N$}\\
		\State Find the smallest integer $j\ge 0$ such that the estimated local Lipschitz constant $L\k=L\km\cdot 2^{j}$ satisfies
		\begin{align}
			\label{Alg:CGSLS:approx}
			f(y\k) \le & f(z\k) + \ang{\nabla f(z\k), y\k - z\k} + \frac{L\k}{2}\norm{y\k - z\k}^2 +\frac{\varepsilon}{2}\g\k,\text{ where }
			\\
			\label{Alg:CGSLS:gamma}
			\gamma\k = &\left\{\begin{array}{ll}
				1 & \text{ if }k=1 \\
				\text{Positive solution to } L_k\gamma_k^3 = \G\km(1-\g\k)& \text{ if }k\ge 2
			\end{array} \right.
			\\
			\label{eq:GammaLg3}
			\Gamma_k = & L_k\gamma_k^3
			\\
			z\k = &(1-\g\k)y\km + \g\k x\km\label{Alg:CGSLS:zk}
			\\
			x\k = &\text{CndG}(\nabla f(z_k), x_{k-1}, \beta_k, \eta_k) \label{Alg:CGSLS:xk}
			\\
			y\k = &(1-\g\k)y\km +\g\k x\k\label{Alg:CGSLS:y}
		\end{align}
		\State Terminate the loop if
		\begin{align}
			\label{Alg:CGSLS:minep}
			f(y\k) - \min_{x\in \X}\xi\k(x)\le \varepsilon\text{
				where function $\xi_k(\cdot)$ is defined by }
			\xi\k(x):=(1-\g_k)\xi\km(x)+\g_k({f(z_k)+\ang{\nabla f(z_k),x-z_k}}).
		\end{align}
		\EndFor
		\State Output approximate solution $y_k$ at the termination of the above for-loop.
		\\
		\ \\
		\noindent \textbf{procedure} $u^+=$ CndG$(g,u,\beta,\eta)$
		\begin{enumerate}
			\item Set $u_1=u$ and $t=1$.
			\item Let $v_t$ be the optimal solution for the subproblem of
			\begin{align}
				\label{Alg:CGS:linApp}
				V_{g,u,\beta}(u_t):=\max_{x\in \X}\ang{g+\beta (u_t-u),u_t-x}
			\end{align}
			\item If $V_{g,u,\beta}(u_t)\le \eta$, set $u^+=u_t$ and terminate the procedure.
			\item Set $u_{t+1}=(1-\alpha_t)u_t+\alpha_t v_t$ where
			$
			\alpha_t=\min\cbr{1,\ang{\beta (u-u_t)-g,v_t-u_t}}/({\beta\norm{v_t-u_t}^2})
			$
			\item Set $t\leftarrow t+1$ and go to step 2.
		\end{enumerate}
		\noindent \textbf{end procedure}
		
	\end{algorithmic}
\end{algorithm}

A few remarks are in place for the proposed CGS-ls algorithm. 
First, the CndG procedure is exactly the same as the one described in the CGS method in \cite{lan2016conditional}. Noting the termination criterion of the CndG procedure, we can observe that the update $x_k$ computed by the CndG procedure satisfies
\begin{align}\label{optCondition:threshold}
	\ang{\gf (z\k)+\beta\k(x\k-x\km),x\k - x} \le \eta\k,\ \forall x\in \cX,
\end{align}
where $\eta_k$ is an accuracy parameter and the $\beta_k$ is a stepsize parameter whose values will be described in the sequel. Note that when the accuracy $\eta_k\equiv 0$ (this is only the ideal case; in practice the CndG procedure will never terminate when $\eta_k$ is set to $0$), then $x_k$ is the exact optimal solution to the problem
\begin{align}
	\label{eq:subproblem}
	\min_{x\in \cX}\langle \nabla f(z_k),x\rangle + \frac{\beta_k}{2}\|x - x\km\|^2.
\end{align}
In such case, the iterates \eqref{Alg:CGSLS:zk}, \eqref{Alg:CGSLS:xk}, and \eqref{Alg:CGSLS:y} becomes the iterates for the accelerated gradient method (see, e.g., \cite{nesterov2004introductory}). Consequently, CGS-ls reduces to the accelerated gradient method with backtracking linesearch. 

Second, when $L_k\equiv L$ where $L$ satisfies condition \eqref{eq:L}, then CGS-ls reduces to a version of CGS method in \cite{lan2016conditional} with $\Gamma_k = L\gamma_k^3$. The concept behind the CGS method is to use a version of conditional gradient method to solve the possibly sophisticated projection subproblem described in \eqref{eq:subproblem}. The theoretical performance limit is achieved through proper choice of the accuracy parameter $\eta_k$.
Note that the convergence analysis of the choice of $\Gamma_k$ and $\gamma_k$ in our proposed CGS-ls method is not discussed in \cite{lan2016conditional}. In the sequel, we will show that the choice $\Gamma_k = L_k\gamma_k^3$ with a backtracking strategy for $L_k$ yields our desired convergence result. 

Third, the termination criterion \eqref{Alg:CGSLS:minep} of the CGS-ls method is based on the linear lower bound function $\xi_k(x)$. In the sequel, we will prove that $\xi_k(x)\le f(x)$ for all $x\in \cX$. Consequently, whenever the termination criterion is satisfied, we have
\begin{align}
	\label{eq:certificate}
	f(y_k) - f^* = f(y_k) - f(x^*) \le f(y_k) - \xi_k(x^*) \le f(y_k) - \min_{x\in\cX}\xi_k(x) \le\varepsilon
\end{align}
where $x^*$ is an optimal solution to problem \eqref{eq:problem}. The above relation certificates that $y_k$ is an approximate solution to problem \eqref{eq:problem}. Such certification strategy for approximate solutions is previously discussed in \cite{nemirovski2010accuracy} for convex optimization problems and implemented in accelerated gradient methods (see, e.g., \cite{nesterov2015universal}). In the sequel, we will prove that the termination criterion is satisfied with at most $\MO(\sqrt{L/\varepsilon})$ gradient evaluations. One alternative termination criterion is the Wolfe gap
\begin{align}
	\max_{x\in\cX}\langle\nabla f(y_k), y_k - x\rangle\le \varepsilon
\end{align}
which is widely employed in the literature of conditional gradient methods. When the Wolfe gap termination criterion is satisfied, we also have $f(y_k) - f^*\le \varepsilon$ due to the convexity of $f$:
\begin{align}
	\label{eq:WolfeGap}
	f(y_k) - f^* = f(y_k) - f(x^*) \le -\langle \nabla f(y_k), x^*- y_k\rangle \le \max_{x\in\cX}\langle\nabla f(y_k), y_k - x\rangle\le \varepsilon.
\end{align}
However, one can only show that the termination criterion through Wolfe gap is satisfied with significantly worse number $\MO(L/\varepsilon)$ of gradient evaluations (see, e.g., \cite{jaggi2013revisiting}). 

Fourth, it is necessary to point out the backtracking strategy we implement in the proposed CGS-ls method. During implementation, we compute the estimated $L_k$ in the following way: we start with $L_k=L\km$ and compute $\gamma_k$, $\Gamma_k$, $z_k$, $x_k$, and $y_k$ in \eqref{Alg:CGSLS:gamma}, \eqref{eq:GammaLg3}, \eqref{Alg:CGSLS:zk}, \eqref{Alg:CGSLS:xk}, and \eqref{Alg:CGSLS:y}. After the computation, we verify whether condition \eqref{Alg:CGSLS:approx} is satisfied. If not, we will multiply $L_k$ by $2$ and backtrack all the values again. Such backtracking procedure stops when the Lipschitz condition \eqref{Alg:CGSLS:approx} is satisfied. It should be noted that we can derive from the convexity of $f$ and the Lipschitz condition \eqref{eq:L} that
\begin{align}
	f(y)\le f(z) + \langle \nabla f(z), y- z\rangle + \frac{L}{2}\|y - z\|^2,\ \forall y,z\in\cX.
\end{align}
Therefore, if $L_0\ge L_{min}$, where $L_{min}$ is the smallest Lipschitz constant that satisfies the Lipschitz condition \eqref{eq:L}, then we have $L_k\equiv L_0\ge L_{min}$, and CGS-ls reduces to CGS with Lipschitz constant $L_0$ and parameter $\Gamma_k = L_0\gamma_k^3$. If $L_0<L_{min}$, it is straightforward to observe that the number of backtracking required throughout the entire iterates of the CGS-ls method is 
$\left\lceil\log({2L_{min}}/{L_0})\right\rceil$.
This is because that whenever $L_k\ge 2L_{min}$, then the condition \eqref{Alg:CGSLS:approx} is always satisfied. Summarizing the above description of $L_k$'s and accounting for both the cases $L_0\ge L_{min}$ and $L_0<L_{min}$, we have the following relation:
\begin{align}
	\label{eq:Lrelation}
	L_0\le L_1\le \cdots \le L_k\le \max\{2L_{min}, L_0\}.
\end{align}

Finally, in order to compute $\gamma_k$ when $k\ge 2$ it suffices to solve the positive root to a cubic polynomial equation $L_k\gamma_k^3=\Gamma\km(1-\gamma_k)$. It is easy to verify that
\begin{align}
	\label{eq:gamma:actual}
	\gamma_k = \sqrt[3]{\frac{\Gamma\km}{2L_k}+\frac{\Gamma\km}{L_k}\sqrt{\frac{1}{4}+\frac{\Gamma\km}{27L_k}}} + \sqrt[3]{\frac{\Gamma\km}{2L_k}-\frac{\Gamma\km}{L_k}\sqrt{\frac{1}{4}+\frac{\Gamma\km}{27L_k}}}\in(0,1),\ \forall k\ge 2.
\end{align}
is the unique positive real root we are looking for through the cubic formula. Here, to prove that $\gamma_k\in(0,1)$ for all $k\ge 2$, note that when $\Gamma\km/L_k\in (0,1)$, from the above cubic formula description of $\gamma_k$ we have $\gamma_k>0$. Also, applying $\Gamma\km/L_k\in (0,1)$ to the relation $L_k\gamma_k^3=\Gamma\km(1-\gamma_k)$ we have $1-\gamma_k> \gamma_k^3>0$. Therefore $\Gamma\km/L_k\in(0,1)$ leads to $\gamma_k\in(0,1)$. Moreover, noting from the description of $L_k$ in Algorithm \ref{Alg:CGSLS} we have $L_k\ge L\km$, hence $\gamma_k\in(0,1)$ implies that $\Gamma_k/L\kp\in(0,1)$. Therefore, applying induction we can prove that $\gamma_k\in (0,1)$ for all $k\ge 2$. Also, note that $\gamma_1=1$ in \eqref{Alg:CGSLS:gamma}. Consequently, $y_k$ in \eqref{Alg:CGSLS:y} is the convex combination of $y\km$ and $x_k$; noting that $x_k\in\cX$ and that $y_0=x_0\in\cX$ we can have that the approximation solution $y_k\in\cX$.


\section{Theoretical analysis}
In this section we perform the convergence analysis of the proposed CGS-ls algorithm in Algorithm \ref{Alg:CGSLS}. We begin with two technical results that will be used in the analysis.

\begin{lemma}\label{bounding:plusMinus}
	Suppose that $\{\lambda_i\}_{i\ge 1}$ and $ \{a_i\}_{i\ge 0}$ are two sequences of nonnegative real numbers, in which the sequence $\{\lambda_i\}_{i\ge 1}$ is non-decreasing. For any fixed $k$, we have
	\begin{align*}
		\sum_{i=1}^{k}\lambda_i(a_{i-1}-a_i) \le \lambda_k \max_{0\le t\le k}a_t.
	\end{align*}
\end{lemma}

\begin{proof}
	%
	Since $\lambda_i \ge \lambda\im\ge 0$ for all $i=2,3,\ldots$, we have immediately that
	\begin{align*}
		\sum_{i=1}^{k}\lambda_i(a_{i-1}-a_i) &= \lambda_1 a_0 + \sum_{i=2}^{k}(\lambda_i - \lambda_{i-1})a_{i-1} - \lambda_k a_k \le \lambda_1\max_{0\le t\le k}a_t + \sum_{i=2}^{k}(\lambda_i - \lambda_{i-1})\max_{0\le t\le k}a_t= \lambda_k \max_{0\le t\le k}a_t.
	\end{align*}
\end{proof}

%
\begin{lemma}
	\label{lem:Sum}
	In Algorithm \ref{Alg:CGSLS}, suppose that $\gamma_1=1$, $\gamma_k\in (0,1)$, $k=2,3,\ldots$, and the value of $\Gamma_k$ satisfies
	\begin{align}
		\label{def:Gamma}
		\G\k:=\left\{\begin{tabular}{ll}
			$1$, & $k=1$,\\
			$\G\km(1-\g\k),$ &$k\ge 2$.
		\end{tabular}\right.
	\end{align}
	If the sequence $\{\delta_k\}_{k\ge 1}$ satisfies
	\begin{align}
		\label{eqDeltaCond}
		\delta_k\le (1-\gamma_k) \delta\km + B_k,\ k=1,2,\ldots,
	\end{align}
	then for any $k\ge 1$ we have
	\begin{align}
		\label{eqDeltaBound}
		\delta_k\le \Gamma_k\sum_{i=1}^k\frac{B_i}{\Gamma_i}.
	\end{align}
	In particular, the above inequality becomes equality when the relations in \eqref{eqDeltaCond} are all equality relations.
\end{lemma}

\begin{proof}
	The result follows from dividing both sides of \eqref{eqDeltaCond} by $\Gamma_k$ and then summing up the resulting inequalities or equalities.
\end{proof}

A few remarks are in place regarding the above lemma. First, by the descriptions of $\gamma_k$ and $\Gamma_k$ in Algorithm \ref{Alg:CGSLS}, the condition \eqref{def:Gamma} is clearly satisfied. Second, applying the above lemma with $\delta_k\equiv 1$ and $B_k=\gamma_k$ we have the following equality:
\begin{align}
	\label{eq:sumone}
	\Gamma_k\sum_{i=1}^k\frac{\gamma_i}{\Gamma_i} = 1.
\end{align}
Similarly, applying the above lemma to the definition of the lower bound function $\xi_k(\cdot)$ in \eqref{Alg:CGSLS:minep} (with $\delta_k=\xi_k(x)$ and $B_k=\g_k({f(z_k)+\ang{\nabla f(z_k),x-z_k}})$) we also have
\begin{align}
	\label{eq:xi:sum}
	\xi_k(x) = \Gamma_k\sum_{i=1}^{k}\frac{\gamma_i}{\Gamma_i}({f(z_i)+\ang{\nabla f(z_i),x-z_i}}).
\end{align}

%

We are now ready to analyze the convergence of the proposed CGS-ls algorithm. Theorem \ref{CGSLS:mainthm} below describes the main convergence property of  Algorithm \ref{Alg:CGSLS}. 

\begin{theorem}
	\label{CGSLS:mainthm}
	Suppose that the parameters in Algorithm \ref{Alg:CGSLS} satisfy $\beta\k\ge L\k\g\k$ for all $k$. Then we have
	\begin{align*}
		f(y\k) - \xi\k(x) &\le \frac{\varepsilon}{2} + \G\k\sum_{i=1}^{k}\frac{\g_i\beta_i}{2\G_i}\pr{\norm{x-x_{i-1}}^2 - \norm{x-x_i}^2} +\G\k\sum_{i=1}^{k}\frac{\g_i\eta_i}{\G_i} 
		,\ \forall x\in\cX.
	\end{align*}
\end{theorem}

\begin{proof}
	Let us fix any $x\in \cX$. In order to prove the result, we will estimate a lower bound of $\xi_k(x)$. From the description of $y_k$ in \eqref{Alg:CGSLS:y} we observe that
	$
	\gamma_k(x_k - z_k) = (y_k- z_k) - (1-\gamma_k)(y\km - z_k)
	$.
	Applying such observation to the description of $\xi_k(x)$ in \eqref{eq:xi:sum} we have
	\begin{align*}
		\frac{1}{\Gamma_k}\xi\k(x) &
		= \sum_{i=1}^{k}\frac{1}{\G\i}\br{\g\i f(z\i)+\g\i\ang{\gf (z_i),x-x_i} + \ang{\gf (z\i),\g\i(x\i-z\i) }}.
		\\
		&= \sum_{i=1}^{k}\frac{1}{\G\i}[ f(z\i)+ \ang{\gf (z_i),y\i-z\i} - (1-\g\i)\pr{f(z\i)+ \ang{\gf (z\i),y\im - z\i }}+\g\i\ang{\gf (z\i),x-x\i}].
	\end{align*}
	We make three observations in the above equation. First, by \eqref{Alg:CGSLS:approx} we have
	\begin{align*}
		f(z\i)+\ang{\gf (z\i),y\i-z\i} &\ge f(y\i)-\frac{L\i}{2}\norm{y\i-z\i}^2 - \frac{\varepsilon}{2}\g\i= f(y\i) - \frac{L\i\g\i^2}{2}\norm{x\i-x\im}^2 - \frac{\varepsilon}{2}\g\i.
	\end{align*}
	Here the last equality is from the descriptions of $z_k$ and $y_k$ in \eqref{Alg:CGSLS:zk} and \eqref{Alg:CGSLS:y} respectively.
	Second, by the convexity of $f$ we have
	\begin{align*}
		-\pr{f(z\i)+\ang{\gf (z\i),y\im-z\i}}\ge -f(y\im).
	\end{align*}
	Third, by the stopping criterion of the CndG procedure in \eqref{optCondition:threshold} and our assumption that $\beta\k\ge L\k\g\k$ for all $k$, we have
	\begin{align*}
		\g\i\ang{\gf (z\i),x-x\i}\ge\ & \gamma_i\beta_i\langle x_i - x\im, x_i - x\rangle - \gamma_i\eta_i
		\\
		= &\ -\frac{\g\i\beta\i}{2}\pr{\norm{x-x\im}^2 - \norm{x\i-x\im}^2 - \norm{x-x\i}^2}-\g\i\eta\i
		\\
		\ge & -\frac{\g\i\beta\i}{2}\pr{\norm{x-x\im}^2 - \norm{x-x\i}^2}-\g\i\eta\i + \frac{L_i\gamma_i^2}{2}\|x_i - x\im\|^2.
	\end{align*}
	Applying the above three observations and recalling that $\gamma_1=1$ and $\gamma_k\in(0,1)$ for all $k\ge 2$ in \eqref{Alg:CGSLS:gamma} and \eqref{eq:gamma:actual} respectively, we obtain that
	\begin{align*}
		\frac{1}{\Gamma_k}\xi\k(x)&\ge \sum_{i=1}^{k}\frac{1}{\G\i}\left[f(y\i)  -(1-\g\i)f(y\im) -\frac{\g\i\beta\i}{2}\pr{\norm{x-x\im}^2 -  \norm{x-x\i}^2} - \frac{\varepsilon}{2}\g\i- \g\i\eta\i\right].
	\end{align*}
	In the above result, noting from the relation \eqref{def:Gamma} between $\gamma_k$ and $\Gamma_k$ and the fact that $\gamma_1=1$, we have
	\begin{align*}
		\sum_{i=1}^{k}\frac{1}{\G\i}f(y\i)-\frac{1-\g\i}{\G\i}f(y\im) =\frac{f(y\k)}{\G\k}.
	\end{align*}
	We conclude the theorem immediately by combining the above two equations and using the relation \eqref{eq:sumone}.
\end{proof}

\begin{corollary}\label{CGSLS:cor:parSetUp3}
	Suppose that the parameters of Algorithm \ref{Alg:CGSLS} are set to
	\begin{align}\label{CGSls:params}
		\b\k =L\k\g\k\text{ and } \eta\k = \frac{L\k\g\k D^2}{k},
	\end{align}
	where $D$ is any constant that estimates the diameter 
	$D_{\cX}:=\max_{x,y\in\cX}\|x-y\|$
	of $\cX$. 
	Then Algorithm \ref{Alg:CGSLS} terminates with an $\varepsilon$-approximate solution $y_k$ after $k\ge N_{grad}$ gradient evaluations, in which
	\begin{align}
		\label{eq:Ngrad}
		N_{grad}:=
		C\sqrt{\frac{\max\{2L_{min},L_0\}D_{\cX}^2}{\varepsilon}}
		, \text{ where }\sqrt{\frac{27}{2}  + \frac{27D^2}{D_{\cX}^2}}\sqrt[6]{\frac{\max\{2L_{min},L_0\}}{L_0}}.
	\end{align}
	At termination, the total number of linear objective optimization (the problem in \eqref{Alg:CGS:linApp}) is bounded by 
	\begin{align}
		\label{eq:Nlin}
		N_{lin}:=\frac{6D_{\cX}^2}{D^2}\frac{C^2\max\{2L_{min},L_0\}D_{\cX}}{\varepsilon} + C\sqrt{\frac{\max\{2L_{min},L_0\}D_{\cX}^2}{\varepsilon}}.
	\end{align}
	Here $L_{min}$ is the smallest Lipschitz constant that satisfies the Lipschitz condition \eqref{eq:L} of the gradient $\nabla f$.
\end{corollary}

\begin{proof}
	Applying Theorem \ref{CGSLS:mainthm} with the parameters described in \eqref{CGSls:params}, and noting from the description of $\gamma_k$ and $\Gamma_k$ in Algorithm \ref{Alg:CGSLS} that $\G\k=L\k\g\k^3 = (1-\g\k)\G\km$, we have
	\begin{align}
		f(y_k) - \xi_k(x) \le \frac{\varepsilon}{2} + \G\k\sum_{i=1}^{k}\frac{1}{2\g_i}\pr{\norm{x-x_{i-1}}^2 - \norm{x-x_i}^2} +\G\k\sum_{i=1}^{k}\frac{D^2}{i\g_i},\ \forall x\in \cX
	\end{align}
	Since $\gamma_k\in(0,1)$ for all $k\ge 2$ (see \eqref{eq:gamma:actual}), we observe that $L_k\gamma_k^3 = \G\k= (1-\g\k)\G\km<\G\km = L\km\gamma\km^3$ for all $k\ge 2$. using this observation and noting from \eqref{eq:Lrelation} that $L_k\ge L\km$, we have $\gamma_k^3<\gamma\km^3$. Consequently, the sequence $\{1/\gamma_i\}_{i\ge 1}$ at the right hand side of the above estimate of $f(y_k)-\xi_k(x)$ is an increasing sequence. Applying Lemma \ref{bounding:plusMinus} we have
	\begin{align}
		\label{eq:est:fxidiff}
		f(y_k) - \xi_k(x) \le \frac{\varepsilon}{2} + \frac{\G\k}{2\g_k}\max_{0\le i\le k}\norm{x-x_i}^2 +\G\k\sum_{i=1}^{k}\frac{D^2}{i\g_i}\le  \frac{\varepsilon}{2} + \frac{\G\k}{2\g_k}D_{\cX}^2 +\G\k\sum_{i=1}^{k}\frac{D^2}{i\g_i}.
	\end{align}
	Here in the last inequality we use the definition of diameter $D_{\mathcal{X}}$. We will estimate the right most side of the above relation.
	
	Using the relation $\G\k=L\k\g\k^3 = (1-\g\k)\G\km$ again, we have
	\begin{align*}
		\sqrt[3]{\frac{1}{\G\k}} - \sqrt[3]{\frac{1}{\G\km}} &= \frac{\frac{1}{\G\k}-\frac{1}{\G\km}}{\sqrt[3]{\frac{1}{\G\k^2}}+\sqrt[3]{\frac{1}{\G\k\G\km}}+\sqrt[3]{\frac{1}{\G\km^2}}} = \frac{\frac{\g\k}{\G\k}}{\sqrt[3]{\frac{1}{\G\k^2}}+\sqrt[3]{\frac{1}{\G\k\G\km}}+\sqrt[3]{\frac{1}{\G\km^2}}}.
	\end{align*}
	Noting that $\gamma_k\in(0,1)$ (see \eqref{eq:gamma:actual}) and recalling that $\G\k=(1-\g\k)\G\km$ we have $\Gamma_k\le \Gamma\km$. Therefore, we have
	\begin{align}
		\sqrt[3]{\frac{1}{\G\k^2}}\le \sqrt[3]{\frac{1}{\G\k^2}}+\sqrt[3]{\frac{1}{\G\k\G\km}}+\sqrt[3]{\frac{1}{\G\km^2}}\le 3\sqrt[3]{\frac{1}{\G\k^2}}
	\end{align}
	Recalling that $\G\k=L\k\g\k^3$, the above two relations imply that
	\begin{align}
		\frac{1}{\sqrt[3]{L_k}}\ge \sqrt[3]{\frac{1}{\G\k}} - \sqrt[3]{\frac{1}{\G\km}} \ge \frac{1}{3\sqrt[3]{L_k}}.
	\end{align}
	Here, recalling the relations of $L_k$'s in \eqref{eq:Lrelation}, we have
	\begin{align}
		\frac{1}{\sqrt[3]{L_0}}\ge \sqrt[3]{\frac{1}{\G\k}} - \sqrt[3]{\frac{1}{\G\km}} \ge \frac{1}{3\sqrt[3]{\max\{2L_{min},L_0\}}}.
	\end{align}
	Summing the above relation from $1$ to $k$ we obtain that
	\begin{align}
		\frac{k-1}{\sqrt[3]{L_0}}\ge \frac{1}{\sqrt[3]{\G\k}} - \frac{1}{\sqrt[3]{\G_1}}\ge \frac{k-1}{3\sqrt[3]{\max\{2L_{min},L_0\}}}.
	\end{align}
	Recalling from \eqref{eq:Lrelation} that $L_0\le L_1\le \max\{2L_{min},L_0\}$ and noting from \eqref{Alg:CGSLS:gamma} and \eqref{eq:GammaLg3} that $\Gamma_1=L_1$, the above becomes 
	\begin{align}
		\frac{k}{\sqrt[3]{L_0}}\ge \frac{1}{\sqrt[3]{\G\k}} 
		\ge \frac{k-1}{3\sqrt[3]{\max\{2L_{min},L_0\}}} + \frac{1}{\sqrt[3]{\max\{2L_{min},L_0\}}} > \frac{k}{3\sqrt[3]{\max\{2L_{min},L_0\}}},
	\end{align}
	i.e.,
	\begin{align}\label{CGSLS:thm:qubGam:leq}
		\frac{L_0}{k^3}\le \G\k \le \frac{27 \max\{2L_{min},L_0\}}{k^3}.
	\end{align}
	Using the first inequality above and recalling the relations $\Gamma_k = L_k\gamma_k^3$ and $L_0\le L_k$ we have
	\begin{align}\label{CGSLS:thm:qubGam:geq}
		\gamma_k \ge \sqrt[3]{\frac{L_0}{L_kk^3}} \ge	\frac{1}{k}\sqrt[3]{\frac{L_0}{\max\{2L_{min},L_0\}}}.
	\end{align}
	Applying the above two results to \eqref{eq:est:fxidiff}, we conclude that
	\begin{align}
		& f(y_k) - \xi_k(x) 
		\\
		\le & \frac{\varepsilon}{2} + \frac{27\max\{2L_{min},L_0\}D_{\cX}^2}{2k^2}\cdot\sqrt[3]{\frac{\max\{2L_{min},L_0\}}{L_0}}  + \frac{27\max\{2L_{min},L_0\}D^2}{k^3}\sum_{i=1}^{k}\sqrt[3]{\frac{\max\{2L_{min},L_0\}}{L_0}}
		\\
		= & \frac{\varepsilon}{2} + \frac{\max\{2L_{min},L_0\}D_{\cX}^2}{k^2}\left[\frac{27}{2}  + \frac{27D^2}{D_{\cX}^2}\right]\sqrt[3]{\frac{\max\{2L_{min},L_0\}}{L_0}},\ \forall x\in \cX.
	\end{align}
	
	Noting the above result and \eqref{eq:certificate}, we conclude that the proposed CGS-ls algorithm will terminate with an $\varepsilon$-approximate solution after $k\ge N_{grad}$ iterations, where $N_{grad}$ is defined in \eqref{eq:Ngrad}. Also, by Theorem 2.2(c) in \cite{lan2016conditional} and our parameter setting \eqref{CGSls:params}, the total number of linear objective optimization that is performed in the $k$-th call to the CndG procedure is bounded by
	\begin{align}
		T_k:=\ceil{\frac{6\beta_k D_{\mathcal{X}}^2}{\eta_k}} = \ceil{\frac{6 D_{\mathcal{X}}^2}{D^2}k}
	\end{align}
	Therefore, at termination the total number of linear objective optimization that is performed by Algorithm \ref{Alg:CGSLS} is bounded by
	\begin{align}
		N_{lin}:=\sum_{i=1}^{N_{grad}}T_i \le \sum_{i=1}^{N_{grad}}T_i\left(\frac{6 D_{\mathcal{X}}^2}{D^2}i+1\right) = \frac{6 D_{\mathcal{X}}^2}{D^2}N_{grad}^2 + N_{grad}.
	\end{align}
	Substituting the value of $N_{grad}$ in \eqref{eq:Ngrad} we obtain \eqref{eq:Nlin}.
\end{proof}


A few remarks are in place for the above corollary. First, from \eqref{eq:Ngrad} and \eqref{eq:Nlin} we conclude that the proposed CGS-ls method has the same theoretical convergence property as that of the CGS method in \cite{lan2016conditional}. Specifically, to compute an $\varepsilon$-approximate solution, the CGS-ls method reaches the theoretical performance limit by requiring at most $\MO(\sqrt{L/\varepsilon})$ gradient evaluations and $\MO({L/\varepsilon})$ linear objective optimizations, where $L:=\max\{2L_{min},L_0\}$ is a Lipschitz constant that satisfies the Lipschitz condition \eqref{eq:L}. Second, in our parameter setting \eqref{CGSls:params} we need to choose an estimate $D$ for the exact diameter $D_{\cX}$. 
However, as long as $D$ is relatively close to $D_{\cX}$ (e.g., smaller or larger than $D_{\cX}$ but within an order 
of $\MO({1/\sqrt{\varepsilon}})$), our convergence properties will not be affected. Finally, when the initial guess of Lipschitz constant $L_0$ is larger than $L_{min}$, no backtracking linesearch will be performed, and the proposed CGS-ls method becomes a version of the CGS method in \cite{lan2016conditional} that is equipped with a proper termination criterion. However, unlike the CGS method, when the initial guess $L_0$ is significantly smaller than $L_{min}$, the convergence property of the CGS-ls method will not be affected significantly. As an example, if we choose $L_0 = 0.001L_{min}$, then in the convergence result of the above corollary we have $
\sqrt[6]{{\max\{2L_{min},L_0\}}/{L_0}} = \sqrt[6]{2000} \approx 3.5,$
namely, the number of gradient evaluations will be enlarged by a constant of approximately $3.5$. In terms of theoretical convergence, such enlargement will not change the order of the gradient evaluations in terms of its dependence on $1/\varepsilon$. However, in terms of numerical implementation, we have the flexibility of choosing much smaller choice of $L_0$. Moreover, if the smaller choice of $L_0$ is satisfied along the iterates of the CGS-ls method, then we can expect that the practical performance of CGS-ls is much faster than algorithms that use a conservative choice of global Lipschitz constant estimate.

	\section{Numerical results}
In this section we present the results from our numerical experiments. We will compare the performance of the proposed CGS-ls method with that of the conditional gradient (CG) method (described in \eqref{eq:cg} with weights $\lambda_i = 2i/(k(k+1))$) and the CGS method in \cite{lan2016conditional} (parameters follow Corollary 2.3 in the paper). We consider two quadratic optimization problems with different subsets; the first is over the standard spectrahedron and the second is over the convex hull of all Hamiltonian cycles. All numerical experiments are performed on a compute with Intel Core i5 2.7 GHz CPU.

In the first numerical experiment, we consider the optimization problem over the standard spectrahedron:
\begin{align}
	\min_{X\in \text{Spe}_n} f(X):=\frac{1}{2}\norm{\mathcal{A}X-B}_F^2\quad \text{where}\quad \text{Spe}_n:=\{X\in\R^{n\times n}: \text{Tr}(X)=1,\ X\succeq 0 \}.
\end{align}
Here $\mathcal{A}:\R^{n\times n}\to \R^{m}$ is a linear operator, $\|\cdot\|_F$ is the Frobenius norm, and the feasible set $\text{Spe}_{n}$ is a standard spectrahedron. Note that the linear objective optimization over the spectrahedron can be solved by computing a maximum eigenvalue problem (see, e.g., \cite{harchaoui2015conditional}). For each random instance in this numerical experiment, we generate $\mathcal{A}$ first by equivalently generating a $m\times n^2$ matrix with $20\%$, $60\%$ or $80\%$ nonzero entries that follow i.i.d. standard normal distribution. We then generate $B=\mathcal{A}U\Sigma U^\top$ where $U$ is a random orthogonal matrix and $\Sigma$ is a diagonal matrix with uniformly random entries between 0 and 1 (normalized afterwards so that they sum to $1$). Therefore, the optimal value of all generated instances are $0$. 

In the second numerical experiment, we consider the optimization problem over the convex hull of Hamiltonian cycles in \cite{lan2017conditional}:
\begin{align}
	&\min_{x\in \mathcal{H}} f(x):=\frac{1}{2}\norm{Ax-b}_2^2\quad \text{where}\quad \mathcal{H}=\text{conv}\{x\in\R^{n(n-1)/2}: x\text{ is a Hamiltonian cycle} \}.
\end{align}
Here $A:\R^{m\times [n(n-1)/2]}$, $b \in \R^m$, and the feasible set $\mathcal{H}$ is the convex hull of all Hamiltonian cycles in a complete graph with $n$ nodes. We describe any Hamiltonian cycle through a vector of dimension $n(n-1)/2$ (the lower triagular part of the adjacency matrix). Note that the linear objective optimization over $\mathcal{H}$ can be solved by computing the solution to a traveling salesman problem (solved through Gurobi \cite{gurobi}). For each random instance, $A$ is randomly generated with $60\%$ nonzero entries that follow i.i.d. standard uniform distribution. We then generate $b=A(0.8v_1+0.2v_2)$ where $v_1$ and $v_2$ are two Hamiltonian cycles that are generated from random permutations of all nodes. The optimal value of all generated instances are $0$.

\begin{tiny}
	\begin{table}[!t]
		\centering
		\caption{\label{tab:SPE}Comparison of CG, CGS and CGS-ls on the first numerical experiment (minimization over standard spectrahedron).}
		%
		%
		%
		\centering
		
		\begin{tabular}{|c|c|c||c c c||c c c c||c c c c|}
			\hline
			\multicolumn{3}{|c||}{Instance info.} & \multicolumn{3}{c||}{CG} & \multicolumn{4}{c||}{CGS} & \multicolumn{4}{c|}{CGS-ls}\\
			\hline
			m & n & density & iter. & time & obj. & outer & inner & time & obj. & outer & inner & time & obj. \\
			\hline 
			1000 & 100 & .2 & 21908 &  543 & 3e-6 & 264 & 528 & 10 & 3e-6 & 148 &  919 & 13 & 6e-8 \\
			2000 & 100 & .2 & 38861 & 1427 & 3e-6 &  842 & 1871 &  38 & 4e-6 & 232 & 1961 & 26 & 8e-8 \\
			3000 & 100 & .2 & 36402 & 1800 & 5e-6 &  900 & 2368 &  53 & 8e-6 & 219 & 2175 & 26 & 2e-7 \\
			
			\hline
			1000 & 100 & .6 & 45170 & 1800 & 1e-6 & 379 & 758 & 18  & 1e-6 & 307 & 1540 & 24 & 2e-8 \\
			2000 & 100 & .6 & 26855 & 1800 & 1e-5 & 1309 & 2964 &  91 & 2e-6 & 343 & 2578 & 35 & 2e-8 \\
			3000 & 100 & .6 & 18830 & 1800 & 5e-5 & 1410 & 3811 & 138 & 3e-6 & 291 & 2797 & 38 & 1e-7 \\
			
			\hline
			1000 & 100 & .8 & 39448 & 1800 & 1e-6 & 393 & 786 & 21  & 1e-6 & 328 & 1625 & 27 & 1e-8 \\
			2000 & 100 & .8 & 22728 & 1800 & 1e-5 & 1475 & 3361 & 117 & 1e-6 & 446 & 2704 & 42 & 3e-8 \\
			3000 & 100 & .8 & 16055 & 1800 & 8e-5 & 1551 & 4223 & 170 & 3e-6 & 320 & 3360 & 45 & 5e-8 \\
			
			\hline
		\end{tabular}
		

		\ \\
		\vspace{1cm}
		%
		%

		\caption{\label{tab:TSP}Comparison of CG, CGS and CGS-ls on the second numerical experiment (minimization over the convex hull of Hamiltonian cycles).}
		\centering
		\begin{tabular}{|c|c|c||c c c||c c c c||c c c c|}
			\hline
			\multicolumn{3}{|c||}{Instance info.} & \multicolumn{3}{c||}{CG} & \multicolumn{4}{c||}{CGS} & \multicolumn{4}{c|}{CGS-ls}\\
			\hline
			m & n & density & iter. & time & obj. & outer & inner & time & obj. & outer & inner & time & obj. \\
			\hline 
			1e3 & 20 & .6 & 37259 &  596 & 3e-6 &  6246 & 12492 &  197 & 4e-3 & 1108 & 2493 &  39 & 4e-4 \\
			1e3 & 25 & .6 & 48415 &  968 & 1e-5 &  9236 & 18472 &  502 & 4e-3 & 1526 & 3730 &  91 & 5e-4 \\
			1e3 & 30 & .6 & 56254 & 1800 & 1e-5 & 12542 & 25084 &  735 & 4e-3 & 1883 & 4734 & 142 & 4e-4 \\
			1e3 & 35 & .6 & 45355 & 1800 & 2e-5 & 15390 & 30780 & 1323 & 5e-3 & 2210 & 5636 & 233 & 4e-4 \\
			1e3 & 40 & .6 & 27638 & 1800 & 1e-4 & 14445 & 28890 & 1800 & 9e-3 & 2511 & 6335 & 374 & 4e-4 \\
			
			\hline
			1e4 & 15 & .6 & 92429 & 1800 & 4e-6 & 11927 & 23854 &  473 & 5e-3 & 2589 &  6711 & 117 & 7e-4 \\
			1e4 & 20 & .6 & 67756 & 1800 & 1e-5 & 19508 & 39016 &  841 & 5e-3 & 3535 &  8702 & 180 & 3e-4 \\
			1e4 & 25 & .6 & 52912 & 1800 & 8e-5 & 23044 & 46088 & 1800 & 7e-3 & 4529 & 11964 & 364 & 3e-4 \\
			1e4 & 30 & .6 & 31736 & 1800 & 2e-4 & 19272 & 38544 & 1800 & 2e-2 & 5564 & 15378 & 623 & 2e-4 \\
			1e4 & 35 & .6 & 24348 & 1800 & 6e-4 & 14771 & 29542 & 1800 & 5e-2 & 6480 & 17848 & 894 & 2e-4 \\
			
			\hline
			1e5 & 15 & .6 & 18433 & 1800 & 1e-3 & 14924 & 29848 & 1800 & 3e-2 & 8036 & 15825 &  774 & 1e-3 \\
			1e5 & 17 & .6 & 14281 & 1800 & 5e-3 & 12791 & 25582 & 1800 & 7e-2 & 8659 & 18406 &  943 & 5e-4 \\
			1e5 & 20 & .6 &  9715 & 1800 & 3e-3 &  9019 & 18038 & 1800 & 2e-1 & 9277 & 18801 & 1427 & 3e-4 \\
			
			\hline
		\end{tabular}
	\end{table}

	%
	%

\end{tiny}

We report the performance of CG, CGS, and CGS-ls in the above two numerical experiments in Table \ref{tab:SPE} and Table \ref{tab:TSP} respectively. For CG and CGS, we terminate when the Wolfe gap described in \eqref{eq:WolfeGap} is smaller than $0.01$; for CGS-ls, we terminate when the gap describe in \eqref{Alg:CGSLS:minep} is smaller than $0.01$. Consequently, all algorithms terminate either when an approximate solution is certified with $\varepsilon=0.01$ (or when the algorithm runs over 30 minutes). Note that CGS requires the Lipschitz constant of the objective function; we compute them through the maximum eigenvalue of the Hessian of the objective function (the time for computing maximum eigenvalue is not counted towards CGS' computation time). For CGS-ls, we set $L_0=10$ for all instances, and $D=0.005\sqrt{2}$ and $D=0.05\sqrt{n(n-1)/2}$ for the first and second numerical experiments, respectively. We report the running time (in seconds) and the objective value of the approximate solution at termination for all algorithms. For CG, we report its total number of iterations, which is the same as the number of gradient evaluations and linear optimization subproblems. For CGS and CGS-ls, we report the total number of gradient evaluations (denoted ``outer'') and linear optimization subproblems (denoted as ``inner''). 

We make a few remarks from Tables \ref{tab:SPE} and \ref{tab:TSP}. First, among the three algorithms, CG is the simplest to implement but has the worst performance; in most of the instances it could not obtain an approximate solution with Wolfe gap smaller than $0.01$ within the required 30 minute computation time limit. Such behavior is consistent with its theoretical complexity $\MO(L/\varepsilon)$, which is the worst among the three algorithms. Second, CGS-ls has better practical performance than CGS-ls in most instances, although both algorithms have the same theoretical convergence properties. The better practical performance is most likely due to the adaptive estimate of Lipschitz constant, which avoids the potentially conservative Lipschitz constant that CGS may suffer throughout the computation. Finally, it is interesting to observe that most objective values of the approximate solutions computed by all algorithm at termination are much better than the accuracy setup $\varepsilon=0.01$. To the best of our knowledge, it is still unclear in the literature whether there exists termination criterion other than the ones we use in this paper (\eqref{Alg:CGSLS:minep} and \eqref{eq:WolfeGap}) that could guarantee that the approximate solution at termination is an $\varepsilon$-solution while achieving good practical performance. We leave the study of better termination criterion as a future work.

	\begin{acknowledgements}
		Yuyuan Ouyang is partially supported by Office of Naval Research award N00014-20-1-2089.
	\end{acknowledgements}

	%
	%

	\bibliographystyle{spmpsci}      
	\bibliography{yuyuan}   
	
	
\end{document}